\documentclass[11pt,reqno]{amsart}
\usepackage{amssymb}
\usepackage[centertags]{amsmath}
\usepackage{amsthm}

\begin{document}
	\def\me{{\rm meas}}
	
	\newcommand{\co}{\overline{co}\hbox{ }}
	\newcommand\inv[1]{#1\raisebox{1.15ex}{$\scriptscriptstyle-\!1$}}
	\title{\textbf{A family of maximal algebras of Block Toeplitz Matrices}}
	\author{Muhammad Ahsan Khan}
	\address{Abdus Salam School of Mathematical Sciences, GC University, Lahore, Pakistan}
	\email{ahsan.khan@sms.edu.pk,smsahsankhan@gmail.com}

	\begin{abstract}
		The maximal commutative subalgebras containing only Toe\-plitz matrices have been identified as generalized circulants. A similar simple description cannot be obtained for block Toeplitz matrices. We introduce and investigate certain families of 
	 maximal commutative algebras of block Toeplitz matrices.
	\end{abstract}
	
	\maketitle
	
	\newcommand{\Ker}{\mathop{\rm Ker}}
	\newcommand{\degree}{\mathop{\rm deg}}
	
	\numberwithin{equation}{section}
	\newtheorem{theorem}{Theorem}[section]
	\newtheorem{proposition}[theorem]{Proposition}

	\newtheorem{corollary}[theorem]{Corollary}
	\newtheorem{lemma}[theorem]{Lemma}
	
	\theoremstyle{definition}
		\newtheorem{remark}[theorem]{Remark}
	\newtheorem{definition}[theorem]{Definition}
	\newtheorem{example}[theorem]{Example}
	\newtheorem{assumption}[theorem]{A}
\section{Introduction}
\advance\baselineskip by 3pt
Toeplitz matrices are one of the most well-studied and understood classes of structured
matrices, that arise naturally in several fields of
mathematics, as well as applied areas as  signal processing or time series analysis. The subject is many decades old; among the monographs dedicated to the subject  are~\cite{ioh}, \cite{grenand-szego},
and  \cite{widom}.

The related area of block Toeplitz matrices is less studied, one of the reasons being the new difficulties that appear with respect to the scalar case. Besides its theoretical interest, the subject is also important in view of the applications to multivariate control theory. As references for block Toeplitz matrices one can use~\cite{botsil} and~\cite{GGK}.

Multiplication properties of Toeplitz matrices have been discussed in~\cite{shalom, gu-patton}. Since the product of two Toeplitz matrices is not necessarily a Toeplitz matrix, it is interesting to investigate subsets of Toeplitz matrices that are closed under multiplication---that is, algebras of Toeplitz matrices. It is not obvious that such nontrivial algebras exist (besides scalar multiples of the identity); but it turns out that one can identify all such maximal algebras~\cite{shalom}; they are given by the so-called generalized circulants~\cite{davis}. 

The investigation of the similar problem for block Toeplitz matrices has been proposed as an open question in~\cite{shalom}, and it does not appear that much progress has been done since. The general problem of characterizing all maximal algebras of block Toeplitz matrices seems very hard. In the current paper we make a few steps in this direction, describing a certain family of such algebras.

The plan of the paper is the following. We start with a section of preliminaries. Then, as we will be interested in block Toeplitz matrices whose entries are themselves elements of maximal commutative subalgebras of matrices, we discuss in Section~3 several such examples.
In Section~4 we introduce  a certain family of maximal algebras that we will consider and prove some basic results. Finally, in the last three sections we discuss this family corresponding to the different maximal commutative subalgebras considered in Section~3. 

\section{preliminaries}\label{se:prelim}

As usual, $\mathbb{C}$ will stand for the complex plane.   We designate the algebra of all $d\times d $ matrices with entries from $\mathbb{C}$ by $ \mathcal{M}_{d\times d}(\mathbb{C})$. An
algebra is a vector space $\mathcal{A}$ over $ \mathbb{C} $ which is closed to multiplication. A subset $\mathcal{B}$ of an algebra
is called a subalgebra if it is itself an algebra; it is commutative if $ab = ba$ for every $a,b\in \mathcal{B}$. A
subalgebra $\mathcal{B}$ of $ \mathcal{A} $ is maximal commutative if for any commutative subalgebra $ \mathcal{B}_1 $ the inclusion  $ \mathcal{B}\subset\mathcal{B}_1 $
 implies that $ \mathcal{B}=\mathcal{B}_1$.

If $\mathcal{B}$ is a subalgebra of $\mathcal{M}_{d\times d}(\mathbb{C})$, then 
$\mathcal{B}^{'}$ denotes the commutant of $ \mathcal{B} $ (the set of all $ d\times d $ matrices commuting with eevry element of $\mathcal{B}$).
 It is not difficult to see that $\mathcal{B}^{'}$ is also an algebra.

For diagonal matrices we will use the notation
\[
diag
\begin{pmatrix}
a_{1}&a_{2}&\cdots& a_{d}
\end{pmatrix}=
\begin{pmatrix} 
a_{1} &  0     &\ldots  & 0\\
0     &  a_{2} &\ldots  & 0\\
\vdots& \vdots & \ddots & \vdots\\
0&  0  & \ldots & a_d
\end{pmatrix}
\]

Classical \emph{Toeplitz matrices} are square matrices whose entries are constant parallel to the main diagonal; that is, they are of the form
\[
T=(t_{p-q})_{p,q=0}^{n-1} : t_j\in\mathbb{C}.
\]
The linear space of all $ n\times n $ Toeplitz matrices will be denoted by $ \mathcal{T}_n $.

If we replace the scalar entries of Toeplitz matrices by $ d\times d $ matrices, we obtain  \emph{block Toeplitz matrices}. Thus a block Toeplitz matrix is actually an $ nd\times nd $ matrix, but which has been decomposed in $ n^2 $ blocks of dimension $ d $, and  these blocks are constant parallel to the main diagonal. We will use the notation $\mathbf{T}=(T_{p-q})_{p,q=0}^{n-1}$,  for a block Toeplitz matrix, and we will 
 denote  by $\mathcal{T}_{n,d}$ the linear space of all $n\times n$ block Toeplitz matrices whose entries are $d\times d$ matrices; thus 
\[
\mathcal{T}_{n,d}=
\left\lbrace 
(T_{p-q})_{p,q=0}^{n-1}
: T_{j}\in \mathcal{M}_{d\times d}(\mathbb{C})
\right\rbrace.
\]
The scalar case is $\mathcal{T}_n=\mathcal{T}_{n,1}$.

It is well known that the product of two scalar Toeplitz matrices is not necessarily a Toeplitz matrix. 
The main multiplicative properties of $ \mathcal{T}_n $ can be found in~\cite[Theorem 1.3]{shalom}. The next two theorems summarize then in a convenient form.

\begin{theorem}\label{th:shalom1}
	The following statements are equivalent:
	\begin{enumerate}
		\item $ \mathcal{A} $ is a maximal subalgebra of $ \mathcal{T}_n $.
		
		\item $ \mathcal{A} $ is a maximal commutative subalgebra of $ \mathcal{T}_n $.
		
		\item  There exist $ a,b\in\mathbb{C} $, not both zero, such that 
		\begin{equation}\label{generalized circulant}
		\mathcal{A}= \Big\{
		T=
		\begin{pmatrix}
		t_{0}& bt_{1}&\ldots &bt_{n-1}\\
		at_{n-1}& t_{0}& \ldots &bt_{n-2}\\
		\vdots& \vdots&\ddots &\vdots\\
		at_{1}& at_{2}&\ldots &t_{0}
		\end{pmatrix}
		: t_0,\dots, t_{n-1}\in \mathbb{C}
		\Big\}.
		\end{equation}
	\end{enumerate} 
\end{theorem}
It is immediate that the algebra $ \mathcal{A} $ defined by~\eqref{generalized circulant} depends only on the quotient $ a/b $. Shalom denotes, for any complex number $\alpha$,  by $\Pi_{\alpha}$ the set of all Toeplitz matrices of the form  \eqref{generalized circulant}, where $a=1$, $b=\alpha$ and by $\Pi_{\infty}$ the set of all Toeplitz matrices of the form \eqref{generalized circulant}, where $a=0$ and $b=1$ (i.e. the set of all upper triangular Toeplitz matrices). These algebras are called \emph{generalized circulants}.

\begin{theorem}\label{th:shalom2}
	If $ T\in \mathcal{T}_n $ is invertible, then there exists $ \alpha\in\mathbb{C}\cup\{\infty\} $ such that $ T $ as well as $ T^{-1} $ belong to $ \Pi_{\alpha} $.
\end{theorem}

As noted in~\cite{shalom}, there is no satisfactory analog of the above results for block Toeplitz algebras, and a classification of maximal commutative subalgebras seems beyond reach. Our purpose in this paper is to explore a  class of maximal commutative subalgebras of the subspace of block Toeplitz matrices.

%

\section{Maximal commutative subalgebras of matrices}\label{se:block toeplitz}

We have already noted that the class of block Toeplitz matrices is not closed with respect to multiplication. In fact, the next lemma can be proved after some slightly tedious computations.

\begin{lemma}\label{le:condition for product to be Toeplitz}
	Suppose $\mathbf{T}=(T_{p-q})_{p,q=0}^{n-1}$ and $\mathbf{U}=(U_{p-q})_{p,q=0}^{n-1}$ are two block Toeplitz matrices with entries in $ \mathcal{M}_{d\times d}(\mathbb{C}) $. The product  $\mathbf{TU}$ is a block Toeplitz matrix if and only if
	\begin{equation}\label{eq:basic product condition}
	T_{p}U_{q-n}=T_{p-n}U_{q}\quad\text{ for all } p,q=1,2,\cdots n-1.
	\end{equation}
\end{lemma}

The main purpose of the present article is to study certain subalgebras of $\mathcal{T}_{n,d}$. Their form is suggested by the generalized circulants that appear in Theorem~\ref{th:shalom1}; we will see however, first that some care must taken with their definition, and secondly that they do not exhaust all subalgebras of $ \mathcal{T}_{n,d} $.

As a preparation, let us remember that the entries of scalar Toeplitz matrices are complex numbers. 
Therefore, as a first step in the study of subalgebras of block Toeplitz matrices, we will assume that their entries belong to a fixed  maximal commutative subalgebra  of $\mathcal{M}_{d\times d}(\mathbb{C})$, that we will denote by $\mathcal{B}$. There are many examples of maximal commutative subalgebras of $\mathcal{M}_{d\times d}(\mathbb{C})$, and their classification is far from being achieved (see, for instance, \cite{brown1, brown2, brown3}). Here is a non-exhaustive list to which we will refer in the sequel.
\begin{enumerate}
	\item If we fix a basis in $ \mathbb{C}^d $, then the algebra of diagonal matrices $\mathcal{D}$ is maximal commutative.

\item 	The generalized circulant algebras $\Pi_{\alpha}$ defined in section 2 are maximal commutative subalgebras. In particular, $ \Pi_0 $ is the algebra of lower triangular scalar Toeplitz matrices, while  $\Pi_{\infty}$ is the algebra of upper triangular scalar Toeplitz matrices.
	
	\item  Fix positive integers $ \sigma, \tau $, such that $\sigma+\tau=d$.  Consider the family of  matrices that have with respect to the decomposition $\mathbb{C}^{d}=\mathbb{C}^{\sigma}\oplus\mathbb{C}^{\tau}$ the form
	\[
	\mathcal{O}_{\sigma, \tau}=
	\left\lbrace
	\begin{pmatrix}
	\lambda I_{\sigma}& X\\
	0            &\lambda I_{\tau}
	\end{pmatrix}\Bigl|
	\lambda\in\mathbb{C},X\in\mathcal{M}_{\sigma\times\tau}(\mathbb{C})
	\right\rbrace\]
It is shown in~\cite{brown1} that, for $ |\sigma-\tau|\le 1 $, $ 	\mathcal{O}_{\sigma, \tau} $ 
	is  a maximal commutative algebra. We will call it  a \emph{Schur algebra}.

\item 	Suppose that $M$ is a nonderogatory matrix; that is, its minimal polynomial is equal to its characteristic polynomial. Then the algebra $\mathcal{P}(M)$ generated by $M$ is   maximal commutative.
	
\end{enumerate}

These examples show the large variety of maximal commutative subalgebras of matrices. We will use them as main test cases for our development.

Let us note, for further use, the following simple result. Recall that a subalgebra $\mathcal{B}$ is said to be inverse-closed if, whenever $B\in\mathcal{B}$ and $B$ is invertible, it follows that $B^{-1}$ also is in $\mathcal{B}$. 

\begin{lemma}\label{le:inverse closed}
	If $\mathcal{B}$ is a maximal commutative subalgebra of $\mathcal{M}_{d\times d}(\mathbb{C})$, then $\mathcal{B}$ is inverse closed.
\end{lemma}

\begin{proof}
	If  $\mathcal{B}$ is a commutative subalgebra, then $\mathcal{B}\subset\mathcal{B}'  $, and
	$\mathcal{B}$ is a maximal commutative subalgebra if and only if $\mathcal{B}=\mathcal{B}'$.   Suppose then $A\in\mathcal{B}$ is invertible. Since $AB=BA$ implies $A^{-1}B=BA^{-1}$, it follows that $A^{-1}\in \mathcal{B}'=\mathcal{B}$.
\end{proof}

\section{A class of subalgebras}

Fixing a maximal commutative subalgebra $\mathcal{B}$ in $\mathcal{M}_{d\times d}(\mathbb{C})$, we will denote by $ \mathcal{T}_{n,d}(\mathcal{B}) $ the linear space of block Toeplitz matrices with entries in $\mathcal{B}$. We are interested in identifying maximal subalgebras of $ \mathcal{T}_{n,d} $ that are contained in $ \mathcal{T}_{n,d}(\mathcal{B}) $. 

The class we will consider is inspired by the generalized circulants $ \Pi_{\alpha} $ in Section~\ref{se:prelim}. We start by noting that we can also define 
\[
\Pi_{\alpha}= \{ T=(t_{p-q})_{p,q=0}^{n-1} : at_j=bt_{j-n} \text{ for } j=1, \dots, n-1  \}.
\]

Let us then fix $ A,B\in \mathcal{B}' $, and define the family $\mathcal{F}_{A,B}^{\mathcal{B}}$ by
\begin{equation}\label{eq:definition of FAB}
\mathcal{F}_{A,B}^{\mathcal{B}}=
\left\lbrace 
(T_{p-q})_{p,q=0}^{n-1}
:
T_{j}\in\mathcal{B}, AT_{j}=BT_{j-n}, j=1,2,\cdots n-1
\right\rbrace .
\end{equation}

We will use the following simple lemma.

\begin{lemma}\label{le:kernel A and B}
Suppose $A,B\in\mathcal{M}_{d\times d}(\mathbb{C})$  satisfy the condition
\begin{equation}\label{eq1}
	KerA\cap KerB=\{0\}.
\end{equation}  
 If $T$ is any $d\times d$ matrix such that $AT=BT=0$ then $T=0$.
\end{lemma}
\begin{proof}
To prove $T=0$ we will show that $Tx=0$ for all $x\in\mathbb{C}^{d}$. The assumption that $ATx=0$ and $BTx=0$ for all $x\in\mathbb{C}^{d}$ implies that  $Tx\in KerA$ and $Tx\in KerB$. It follows that $Tx\in KerA\cap KerB.$ But since we have $KerA\cap KerB=\{0\}$, it follows that $Tx=0$ and hence $T=0$.
\end{proof}

We obtain then the main result of this section.

\begin{theorem}\label{t1} Suppose $\mathcal{B}$ is a commutative subalgebra of $\mathcal{M}_{d\times d}(\mathbb{C})$, and $A, B\in \mathcal{B}'$. Then
the family $\mathcal{F}_{A,B}^{\mathcal{B}}$ is a commutative linear subspace of $\mathcal{T}_{n,d}(\mathcal{B})$. If $A,B$ satisfy condition~\eqref{eq1}, then $\mathcal{F}_{A,B}^{\mathcal{B}}$ is an algebra. 
\end{theorem}

\begin{proof}
It is easy to see that $\mathcal{F}_{A,B}^{\mathcal{B}}$ is a commutative linear space. Let us then assume that condition~\eqref{eq1} is satisfied.
To show that  $\mathcal{F}_{A,B}^{\mathcal{B}}$  is closed under  matrix multiplication, 
suppose that $\mathbf{T}$ and $\mathbf{U}$ are two arbitrary elements of $\mathcal{F}_{A,B}^{\mathcal{B}}$; so
\begin{equation}\label{eq:FABforTU}
\begin{split}
\mathbf{T}=(T_{p-q})_{p,q=0}^{n-1},&\quad
AT_{j}=BT_{j-n}\quad j=1,2,\cdots,n-1,\\
\mathbf{U}=(U_{p-q})_{p,q=0}^{n-1},&\quad 
AU_{j}=BU_{j-n}\quad j=1,2,\cdots,n-1.
\end{split}
\end{equation}
The entry at the position $(i,j)$ of $\mathbf{T}\mathbf{U}$ is 
\begin{equation}\label{eq:formula entry product}
(\mathbf{T}\mathbf{U})_{i,j}= \sum_{k=0}^{n-1}T_{i-k}U_{k-j},\quad i,j=0,1,\cdots n-1,
\end{equation}
and therefore
\[
(\mathbf{T}\mathbf{U})_{i,j}-(\mathbf{T}\mathbf{U})_{i+1,j+1}=
T_{i-n+1}U_{n-1-j}-T_{i+1}U_{-1-j}.
\]
Multiplying with $ A $ and with $ B $ respectively and using formulas~\eqref{eq:FABforTU},
it follows that
\[
\begin{split}
A((\mathbf{T}\mathbf{U})_{i,j}-(\mathbf{T}\mathbf{U})_{i+1,j+1})&=0,\\
B((\mathbf{T}\mathbf{U})_{i,j}-(\mathbf{T}\mathbf{U})_{i+1,j+1})&=0.
\end{split}
\]
We may then apply Lemma~\ref{le:kernel A and B} to conclude that
\[
(\mathbf{T}\mathbf{U})_{i,j}-(\mathbf{T}\mathbf{U})_{i+1,j+1}=0,
\]
and thus $\mathbf{V}:=\mathbf{T}\mathbf{U}\in \mathcal{T}_{n.d}(\mathcal{B})$. Therefore $ \mathbf{V}=(V_{i-j})_{i,j=0}^{n-1} $.

To prove that $\mathbf{V}\in \mathcal{F}_{A,B}^{\mathcal{B}} $, 
 consider $j=0$ in~\eqref{eq:formula entry product}. Then
 \[
 \begin{split}AV_i&=
  A (\mathbf{T}\mathbf{U})_{i,0}= A( \sum_{k=0}^{n-1}T_{i-k}U_{k} )
 = \sum_{k=0}^{n-1} T_{i-k}AU_{k}\\
& = \sum_{k=0}^{n-1} T_{i-k}BU_{k-n}
 =B(\sum_{k=0}^{n-1} T_{i-k}U_{k-n})= B(\mathbf{T}\mathbf{U})_{i,n}=BV_{i-n}.
 \end{split}
  \]

Therefore $\mathbf{V}\in\mathcal{F}_{A,B}^{\mathcal{B}}$. This proves, along with the earlier fact that $\mathcal{F}_{A,B}^{\mathcal{B}}$ is a linear subspace, that $\mathcal{F}_{A,B}^{\mathcal{B}}$  is an algebra in $\mathcal{T}_{n,d}$. 
\end{proof}

\begin{remark}\label{re:basic condition not necessary}
	Condition~\eqref{eq:definition of FAB} has been shown to imply that $\mathcal{F}_{A,B}^{\mathcal{B}}$ is an algebra. However, the converse is not true; an example will appear in Section~\ref{se:schur} below.
\end{remark}

 The next result concerns  the maximality of $\mathcal{F}_{A,B}^{\mathcal{B}}$ as a commutative algebra in $\mathcal{T}.$
\begin{theorem} \label{th:max com}
	Let $\mathcal{B}$ be a commutative subalgebra of $\mathcal{M}_{d\times d}(\mathbb{C})$. The following assertions are equivalent:
	\begin{enumerate}
		
		\item $\mathcal{B}$ is a maximal commutative subalgebra of $\mathcal{M}_{d\times d}(\mathbb{C})$.
		
			\item The family $\mathcal{F}_{A,B}^{\mathcal{B}}$ is a maximal commutative algebra in $ \mathcal{T}_{n,d} $.

	\end{enumerate}

\end{theorem}
\begin{proof}
	Suppose that $\mathcal{B}$ is a maximal commutative subalgebra of $\mathcal{M}_{d\times d}(\mathbb{C})$.  Suppose that $\mathcal{F}$ is a commutative algebra in $ \mathcal{T}_{n,d} $ such that $\mathcal{F}_{A,B}^{\mathcal{B}}\subset\mathcal{F}$. Let $\mathbf{T}=(T_{p-q})_{p,q=0}^{n-1}$
	be any arbitrary element of $\mathcal{F}$. For an arbitrary $ U\in \mathcal{B}$ define 
	$\mathbf{U}=diag
	\begin{pmatrix}
	U& U&\cdots U
	\end{pmatrix}.
	$
	Then $\mathbf{U}\in\mathcal{F}_{A,B}^{\mathcal{B}}$. Since $\mathcal{F}_{A,B}^{\mathcal{B}}\subset\mathcal{F}$ and $\mathcal{F}$ is a commutative algebra in $\mathcal{T}$, we have  $\mathbf{T}\mathbf{U}=\mathbf{U}\mathbf{T}$.
	Comparing corresponding entries of $\mathbf{T}\mathbf{U}$ and $\mathbf{U}\mathbf{T}$, we obtain $T_{p}U=UT_{p}$ for all $p$, $ |p|\le n-1 $.
	Since $U\in\mathcal{B}$ and $\mathcal{B}$ is a maximal commutative subalgebra of $\mathcal{M}_{d\times d}(\mathbb{C})$, it follows that $T_{p}\in\mathcal{B}$ for all~$p$.

	Now let $\mathbf{T}\in\mathcal{F}$ be any arbitrary element of the form $\mathbf{T}=(T_{p-q})_{p,q=0}^{n-1}$ and consider the fixed element $\mathbf{J}\in\mathcal{F}_{A,B}^{\mathcal{B}}$ defined by 
	\[\mathbf{J}=
	\begin{pmatrix}
	
	0     &  0    &0&\ldots  & A\\
	B     &  0    &0&\ldots  & 0\\
	0&B& 0&\ldots&0\\
	\vdots    & \vdots&\vdots&    \ddots  & \vdots\\
	0   &    0   & \ldots&B & 0
	\end{pmatrix}
	\]
Using formula~\eqref{eq:formula entry product} for $ \mathbf{J}\mathbf{T} $, one obtains
	\begin{equation}\label{eq4}
	(\mathbf{J}\mathbf{T})_{i,j}=
	\begin{cases}
	AT_{n-1-j} &\text{ for }i=0,\quad j=0, \dots, n-1,\\
	BT_{i-1-j} &\text{ for }i\not=0,\quad j=0, \dots, n-1.
	\end{cases}
	\end{equation}
	Since $\mathcal{F}$ is contained in $ \mathcal{T}_{n,d} $, the entries along the main diagonals of $\mathbf{J}\mathbf{T}$ should be equal.
	In particular, 
	\[
	(\mathbf{J}\mathbf{T})_{0,j}=(\mathbf{J}\mathbf{T})_{1,j+1}
	\]
	for $ j=1, \dots, n-2 $.
	By~\eqref{eq4} this means $ AT_{n-1-j}=BT_{-j-1} $, or, by denoting $ i=n-1-j $,
	\[
	AT_i=BT_{i-n}
	\] 
	for $ i=1, \dots, n-1 $.
	Therefore
	  $\mathbf{T}\in\mathcal{F}_{A,B}^{\mathcal{B}}$. Consequently  $\mathcal{F}\subset\mathcal{F}_{A,B}^{\mathcal{B}}$. whence $\mathcal{F}_{A,B}^{\mathcal{B}}$ is a maximal commutative algebra in $\mathcal{T}_{n,d}$.
	  
	Conversely, suppose that $\mathcal{F}_{A,B}^{\mathcal{B}}$ is a maximal commutative algebra in $\mathcal{T}_{n,d}$. If $\mathcal{B}_{1}$ is a subalgebra of $\mathcal{M}_{d\times d}(\mathbb{C})$ such that $\mathcal{B}\subset\mathcal{B}_{1}$, then  $\mathcal{F}_{A,B}^{\mathcal{B}}\subset\mathcal{F}_{A,B}^{\mathcal{B}_{1}}$. 
	The maximality of
$\mathcal{F}_{A,B}^{\mathcal{B}}$ implies that $\mathcal{F}_{A,B}^{\mathcal{B}}=\mathcal{F}_{A,B}^{\mathcal{B}_{1}}$, whence it follows easily that
 $\mathcal{B}=\mathcal{B}_{1}$. Therefore
	$\mathcal{B}$ is a maximal commutative subalgebra of $\mathcal{M}_{d\times d}(\mathbb{C})$. 
\end{proof}

As a consequence of Theorem~\ref{th:max com}, we will consider in the sequel only maximal commutative algebras $ \mathcal{B} $, whence $ \mathcal{B}=\mathcal{B}' $. We will therefore take in the definition of the algebras $\mathcal{F}_{A,B}^{\mathcal{B}}$ matrices $ A,B\in \mathcal{B} $.

These algebras include an  important general class of maximal block Toeplitz algebras, as shown by the next result. 

\begin{theorem}\label{th:invertible entry}
	Let $\mathcal{B}$ a maximal commutative subalgebra of $\mathcal{M}_{d\times d}(\mathbb{C})$ and $\mathcal{A}$ an algebra contained in  $\mathcal{T}_{n,d}(\mathcal{B})$. Suppose $\mathcal{A}$ contains an element  $\mathbf{T}$ such that $T_{r}$ is invertible for some $r\neq0$. Then $\mathcal{A}\subset\mathcal{F}_{A,B}^{\mathcal{B}}$ for some $A$ and $B$. 
\end{theorem}
\begin{proof}
	Let $\mathbf{T}=(T_{p-q})_{p,q=0}^{n-1}$ and $\mathbf{U}=(U_{p-q})_{p,q=0}^{n-1}$ be two elements of $\mathcal{A}$. Since $\mathcal{A}$ is an algebra contained in $\mathcal{T}_{n,d}$,  the product $\mathbf{TU}\in\mathcal{A}$ if and only if 
	\begin{equation}\label{eq6}
	T_{p}U_{q-n}=T_{p-n}U_{q}\quad (q=1,2,\cdots n-1). 
	\end{equation}
	Suppose $r>0$. Since we have assumed that $T_r$ is invertible, it follows from \eqref{eq6} that
	\[
	U_{q-n}=T_{r}^{-1}T_{r-n}U_{q}\quad (q=1,2,\cdots n-1).
	\]
	The subalgebra $\mathcal{B}$ is inverse closed by Lemma~\ref{le:inverse closed}, and so $T_{r}^{-1}T_{r-n}\in\mathcal{B}$.
	It follows then that, if we take $A=I$ and $B=T_{r}^{-1}T_{r-n}$, then $\mathcal{A}\subset\mathcal{F}_{A,B}^{\mathcal{B}}$. A simlar argument works for $r<0$, finishing the proof.
\end{proof}

\begin{corollary}\label{co:invertible entry}
		Let $\mathcal{B}$ a maximal commutative subalgebra of $\mathcal{M}_{d\times d}(\mathbb{C})$ and $\mathcal{A}$ a maximal subalgebra of $\mathcal{T}_{n,d}(\mathcal{B})$.  Suppose $\mathcal{A}$ contains an element  $\mathbf{T}$ such that $T_{p}$ is invertible for some $p\neq0$. Then $\mathcal{A}=\mathcal{F}_{A,B}^{\mathcal{B}}$ for some $A$ and $B$. 
\end{corollary}

In the next sections we will investigate the algebras 
$\mathcal{F}_{A,B}^{\mathcal{B}}$ corresponding to several cases of maximal subalgebras $ \mathcal{B} $.

\section{Diagonal entries} 

To discuss the cases of entries belonging to the algebra of diagonal matrices with respect to a given basis, we use the next lemma, which describes the structure of such block Toeplitz matrices.

\begin{lemma}\label{l6}
	Suppose $\mathbf{T}=(T_{p-q})_{p,q=0}^{n-1}$ is a block Toeplitz matrix, such that each $T_{j}$ is diagonal. Then there is a change of basis that brings $\mathbf{T}$ into the following form 
	\[\mathbf{T}^{\prime}=diag
	\begin{pmatrix}
	T_{1}^{\prime}&T_{2}^{\prime}&\cdots T_{d}^{\prime}
	\end{pmatrix}
	,\]
	where for every $k=1,2,\cdots d$, $T_{k}^{\prime}$ is a scalar 
	Toeplitz matrix of order $ n $. 
\end{lemma}
\begin{proof} Suppose that the original basis is
	\[
		\mathcal{E}=\{e_{1}^{0},e^{0}_{2
	},\cdots e^{0}_{d},e^{1}_{1},e^{1}_{2},\cdots e^{1}_{d},\cdots e^{n-1}_{1},e^{n-1}_{2},\cdots e^{n-1}_{d}\},
	\]
and $
	T_{j}=diag
	\begin{pmatrix}
	t_{j1}&t_{j2}&\cdots t_{jd}
	\end{pmatrix}
	$  for $j=0,\pm1,\cdots\pm (n-1)$. The new basis is then
	\[
	\mathcal{E}'=\{e_{1}^{0},e_{1}^{1
	},\cdots e_{1}^{n-1},e_{2}^{0},e_{2}^{1},\cdots e_{2}^{n-1},\cdots e_{d}^{0},e_{d}^{1},\cdots e_{d}^{n-1}\}
	\]
	(it is obtained by ``reshuffling'').
	With respect to $\mathcal{E}'$, $\mathbf{T}$ has the required form 
	$
	\mathbf{T}^{\prime}=diag
	\begin{pmatrix}
	T_{1}^{\prime}&T_{2}^{\prime}&\cdots T_{d}^{\prime}
	\end{pmatrix}
	$,
	where, for each $k=1,2,\cdots d$, 
	$ T_{k}^{\prime}=(t_{r-s,k})_{r,s=0}^n$,
 i.e. $T_{k}^{\prime}$ is a Toeplitz matrix of order $ n $. 
\end{proof}

It is clear from the proof that  $ \mathcal{E}' $ depends only on the base $ \mathcal{E} $,  not on the particular matrix $ \mathbf{T} $.

\begin{proposition}	
	Suppose that  $\mathcal{A}$ is any commutative algebra of block Toeplitz matrices contained in
	$\mathcal{X}=
	\left\lbrace 
	(T_{p-q})_{p,q=0}^{n-1
	}\Bigl|T_{j}\in\mathcal{D}
	\right\rbrace 
	$. Then $\mathcal{A}\subset\mathcal{F}_{A,B}^{\mathcal{D}}$ for some $A,B\in\mathcal{D}$.
\end{proposition}
\begin{proof}
	Applying Lemma~\ref{l6}, we obtain a basis $ \mathcal{E}' $ with respect to which any element $ \mathbf{T}\in\mathcal{A} $ has the form
%
%
%
	$
	\mathbf{T}^{\prime}=diag
	\begin{pmatrix}
	T_{1}^{\prime}&T_{2}^{\prime}&\cdots T_{d}^{\prime}
	\end{pmatrix}
	$,
	where for each $k=1,2,\cdots d$ we have $ 
	T_{k}^{\prime}=(t_{r-s,k})_{r,s=0}^{d}$;
	that is, $T_{k}^{\prime}$ is a scalar 
	Toeplitz matrix. Since $ \mathcal{A} $ is closed with respect to multiplication, each of the component blocks has to be closed with respect to multiplication. It follows that 
	  for every $k=1,2,\cdots d, T_{k}^{\prime}\in\Pi_{\alpha_{k}}$ for some $\alpha_{k}\in\mathbb{C}\cup\{\infty\}$. 
	  
	  We may then define 
	  $ B=diag
	  \begin{pmatrix}
	  b_{1}&b_{2}&\cdots& b_{d}
	  \end{pmatrix}$ $ A=diag
	  \begin{pmatrix}
	  a_{1}&a_{2}&\cdots& a_{d}
	  \end{pmatrix}$ and as follows:
	  
	  --- if $ \alpha_k\not=\infty $, then $ a_k=\alpha_k $ and $ b_k=1 $;
	  
	  --- if $ \alpha_k=\infty $, then $ a_k=1 $ and $ b_k=0 $.

Then $KerA\cap KerB=\{0\}$, and it is easily checked that $\mathcal{A}\subset\mathcal{F}_{A,B}^{\mathcal{D}}$.	  
\end{proof}

\begin{corollary}\label{co:diagonal}
	If $\mathcal{A}$ is a maximal subalgebra of $\mathcal{T}_{n,d}(\mathcal{D})$, then $\mathcal{A}=\mathcal{F}_{A,B}^{\mathcal{D}}$ for some $A, B\in \mathcal{D}$. 	
\end{corollary}

\section{Entries in the schur algebras}\label{se:schur}

We start with the following simple lemma.

\begin{lemma}\label{l7}
	An element $T\in\mathcal{O}_{\sigma,\tau}$ is invertible if and only if $\lambda\neq 0$.
\end{lemma}
\begin{proof}
	Let $T\in\mathcal{O}_{\sigma, \tau
	}$ be an element
	of the form 
	\[
	T=
	\begin{pmatrix}
	\lambda I_{\sigma}&  X\\
	0      & \lambda I_{\tau}
	\end{pmatrix}
	.\]
	Then $T$ is invertible if and only if \hbox{det}$(T)=\lambda^{d}\neq0$ if and only if $\lambda\neq 0$. 
\end{proof}

Suppose now that $\mathcal{A}$ is a maximal commutative subalgebra that has its entries in the Schur algebra $\mathcal{O}_{\sigma,\tau}$. If at least one element of $\mathcal{A}$ has an invertible off-diagonal entry, then we can apply Theorem 4.5 and obtain that $\mathcal{A}=\mathcal{F}_{A,B}^{\mathcal{O}_{\sigma,\tau}}$ for some $A,B\in \mathcal{O}_{\sigma,\tau}$. The next theorem shows that there is only one other maximal commutative subalgebra contained in $\mathcal{T}_{n,d}(\mathcal{O}_{\sigma,\tau})$.

\begin{theorem}
	The set $\mathcal{S}$ of  block Toeplitz matrices that have entries in $\mathcal{O}_{\sigma,\tau}$, and all off diagonal entries are noninvertible forms a maximal commutative algebra. It is  of  type $\mathcal{F}_{A,B}^{\mathcal{B}}$ for some $ A,B $ that do not satisfy condition~\eqref{eq1}.
\end{theorem}
\begin{proof}
	First we will show that $\mathcal{S}$ is an algebra. Clearly $\mathcal{S}$ is a linear subspace. we have to show only that it is closed under block matrix multiplication.  For this let $\mathbf{T}=(T_{p-q})_{p,q=0}^{n-1}$ and $\mathbf{U}=(U_{p-q})_{p,q=0}^{n-1}$ be any two arbitrary elements of $\mathcal{S}$ then $\mathbf{TU}\in\mathcal{S}$ if and only if 	
	\[
	T_{p}U_{q-n}=T_{p-n}U_{q}\quad (p,q=1,2,\cdots n-1).
	\]
	Since each $T_{i}$ and $U_{i}$ is noninvertible, lemma \ref{l7} implies that    for every $i\not=0$ $T_{i}$ and $U_{i}$ are strictly upper triangular $2\times 2$ block matrices. Therefore  
	\begin{equation}\label{eq16}
	T_{p}U_{q-n}=T_{p-n}U_{q}=0\quad (p,q=1,2,\cdots n-1),
	\end{equation}  
	whence it follows that $\mathbf{TU}\in\mathcal{S}$. 
	
	We will show that $\mathcal{S}$ is a maximal commutative subalgebra contained in $\mathcal{T}_{n,d}(\mathcal{O}_{\sigma,\tau})$. Suppose  $\mathbf{U}=(U_{p-q})_{p,q=0}^{n-1}$ commutes with $\mathcal{S}$ and denote
	\[U_{k}=
	\begin{pmatrix}
	\lambda_kI_{\sigma}&X_{k}\\
	0   & \lambda_kI_{\tau}
	\end{pmatrix},\quad \lambda_k\in\mathbb{C}\quad, X_{k}\in\mathcal{M}_{\sigma\times\tau}(\mathbb{C}),
	\]
	Suppose, for instance, that $k\ge1$.  
	Take then  $\mathbf{T}=(T_{p-q})_{p,q=0}^{n-1}\in\mathcal{S}$ defined by
	\[T_{p}=
	\begin{pmatrix}
	0 & 0\\
	0   & 0
	\end{pmatrix},\quad
	T_{p-n}=
	\begin{pmatrix}
	0 & Y\\
	0   & 0
	\end{pmatrix},\quad p=1,2,\cdots n-1,
	\] 
	with $ Y$ some nonzero matrix in  $\mathcal{O}_{\sigma,\tau}$.
	Then
	\[
	T_{p}U_{k-n}=0,\quad  T_{p-n}U_{k}=
\begin{pmatrix}
0&\lambda_k\\ 0 & 0
\end{pmatrix},
\quad p=1,2,\cdots n-1.
	\]
	From~\eqref{eq:basic product condition} it follows then that $\lambda_k=0$. A similar argument works for $k<0$.
	Therefore $\mathbf{U}\in\mathcal{S}$, whence $\mathcal{S}$ is maximal commutative. 
	
	To obtain $\mathcal{S}=\mathcal{F}_{A,B}^{\mathcal{O}_{\sigma,\tau}}$, take two linearly independent matrices $ X,Y\in \mathcal{O}_{\sigma,\tau} $, and define
	\[
	A=\begin{pmatrix}
	0&X\\0&0
	\end{pmatrix},\quad B=\begin{pmatrix}
	0&Y\\0&0
	\end{pmatrix}.
	\]
	If $\mathbf{T}=(T_{p-q})_{p,q=0}^{n-1}\in\mathcal{S}$, then 
	$ AT_j=BT_{j-n}=0 $ for all $ j=1, \dots, n $. So $\mathcal{S}\subset \mathcal{F}_{A,B}^{\mathcal{O}_{\sigma,\tau}}  $.
	
	To prove the reverse, take $\mathbf{T}=(T_{p-q})_{p,q=0}^{n-1}\in\mathcal{F}_{A,B}^{\mathcal{O}_{\sigma,\tau}}$. If 
	\[
	T_i=\begin{pmatrix}
	\lambda_i&X_i\\0&\lambda_i
	\end{pmatrix},
	\]
	then 
	\[
	AT_i=\begin{pmatrix}
	0&\lambda_i X\\ 0&0
	\end{pmatrix}, \quad
		BT_{i-n}=\begin{pmatrix}
	0&\lambda_{i-n} Y\\ 0&0
	\end{pmatrix}.
	\]
	Since $ X,Y $ are linearly independent, condition~\eqref{eq:definition of FAB} implies $ \lambda_i=\lambda_{i-n}=0 $ for $ i=1, \dots, n-1 $.
	Therefore $ \mathbf{T} $ is not invertible, so it belongs to $ \mathcal{O}_{\sigma, \tau} $.
\end{proof}
	 
	 It is interesting to note that in this case $\mathcal{F}_{A,B}^{\mathcal{B}}$ is an algebra, although $ A,B $  do not satisfy condition~\eqref{eq1}. Note that we cannot write $\mathcal{S}=\mathcal{F}_{A,B}^{\mathcal{B}}$ with some $ A,B $ satisfying condition~\eqref{eq1}. Indeed, for a noninvertible matrix $ A\in \mathcal{O}_{\sigma, \tau} $ we have $\mathbb{C}^\sigma\oplus 0\in \ker A $. Therefore, to satisfy condition~\eqref{eq1}, at least one of $ A $ and $ B $ has to be invertible. 
	 Suppose then that
	 	\[
	 A=\begin{pmatrix}
	 \lambda&X\\0&\lambda
	 \end{pmatrix},\quad B=\begin{pmatrix}
	 \mu&Y\\0&\mu
	 \end{pmatrix}.
	 \]
	 For a noninvertible $\mathbf{T}=(T_{p-q})_{p,q=0}^{n-1}$, with
	 \[
	 T_i=\begin{pmatrix}
	0&X_i\\0&0
	 \end{pmatrix},
	 \]
	  condition~\eqref{eq:definition of FAB} becomes
	  \[
	  \begin{pmatrix}
	  0&\lambda X_i\\0&0
	  \end{pmatrix}=
	   \begin{pmatrix}
	  0&\mu X_{i-n}\\0&0
	  \end{pmatrix},
	  \]
	  or $\lambda X_i=\mu X_{i-n}  $
	 If at least one of $ \lambda,\mu $ is nonzero, this equality is not satisfied by every $ \mathbf{T}\in\mathcal{S} $.
	 
	 \begin{corollary}\label{co:schur}
	 	If $\mathcal{A}$ is a maximal subalgebra of $\mathcal{T}_{n,d}(\mathcal{O}_{\sigma, \tau} )$, then $\mathcal{A}=\mathcal{F}_{A,B}^{\mathcal{O}_{\sigma,\tau}}$ for some $A, B\in \mathcal{O}_{\sigma, \tau} $. 	
	 \end{corollary}

	 \section{Entries in a singly generated algebra}
	Suppose $M$ is a nonderogatory matrix, with minimal polynomial $p_M$ of degree $\delta$, and $\mathcal{B}$ is the algebra generated by $M$, that is 
	\[
	\mathcal{B}=\mathcal{P}(M):=\{p(M): p \text{ polynomial }\}=
	\{p(M): p \text{ polynomial }, \degree p<\delta\}.
	\]
	The next lemma is well known.
	
	\begin{lemma}\label{le:inv nonderogatory}
		An element $p(M)\in \mathcal{P}(M)$ is invertible if and only if $p$ and $ p_M $ are relatively prime.
	\end{lemma}

As a consequence of Theorem~\ref{th:invertible entry}, we have then the next result.

\begin{corollary}\label{co:inv entry nonderogatory}
Suppose $\mathcal{A}\subset \mathcal{T}$ is an algebra with entries in $\mathcal{P}(M)$. Suppose $\mathcal{A}$ contains an element  $\mathbf{T}$ such that for some $j\neq0$,  $T_{j}=p(M)$ with $p$ and $ p_M $  relatively prime. Then $\mathcal{A}\subset\mathcal{F}_{A,B}^{\mathcal{P}(M)}$ for some $A, B\in \mathcal{P}(M)$. 
\end{corollary}
	
	However, not all maximal subalgebras of $ \mathcal{T}_{n,d} $ are of type $ \mathcal{F}_{A,B}^{\mathcal{P}(M)} $. To see this, we will discuss in the rest of this section 
 the particular case of a nilpotent matrix of order 2; that, we assume that $M\not=0$ and $M^2=0$. 
	
	\begin{theorem}\label{th:counterexample}
			Suppose $M$ is a nonderogatory matrix such that $M^2=0$. The set $\mathcal{A}$ of all block Toeplitz matrices that have entries in $\mathcal{P}(M)$, and all off diagonal entries are noninvertible, form a maximal commutative algebra that is not of the type $\mathcal{F}_{A,B}^{\mathcal{P}(M)}$.
	\end{theorem}

\begin{proof}
	It is easy to see that $\mathcal{A}$ is an algebra: if $\mathbf{T}, \mathbf{S}\in \mathcal{A}$, then any nondiagonal entry of $\mathbf{T} \mathbf{S}$ is a sum of products, and each of the products contains at least one term that has $M$ as a factor, and is thus noninvertible.
	
	To show that $\mathcal{A}$ is maximal commutative, suppose that $\mathbf{T} \mathbf{U}=\mathbf{U} \mathbf{T} $ for all $\mathbf{T}\in\mathcal{A}$. 	
	Take $\mathbf{T}$ with 
	\begin{equation}\label{eq:choice of T1}
		T_p=0, \quad T_{p-n}=M \text{ for } p=1,\dots, n.
	\end{equation}
	Formula~\eqref{eq:basic product condition} tells us that 
	$
	0=M U_q
$
	for all $q=1,\dots, n$, whence it follows that $U_q$ cannot be invertible. Similarly, by taking 
	$\mathbf{T}$ with 
		\begin{equation}\label{eq:choice of T2}
	T_p=M, \quad T_{p-n}=0 \text{ for } p=1,\dots, n,
	\end{equation}
	we obtain $U_{q-n}$ is noninvertible for all $q=1,\dots, n$. Therefore 
	$\mathbf{U}$ has all its off-diagonal entries noninvertible, and therefore it belongs to $\mathcal{A}$.
	
	Suppose now that $\mathcal{A}=\mathcal{F}_{A,B}^{\mathcal{P}(M)}$ for some $A,B\in\mathcal{A}$. In particular, using again $\mathbf{T}$ from~\eqref{eq:choice of T1}  it follows from~\eqref{eq:definition of FAB} that $BM=0$. Similarly, using~\eqref{eq:choice of T2} we obtain $AM=0$. So $A$ and $B$ are noninvertible elements in $\mathcal{P}(M)$, say $A=aM$, $B=bM$. 

If $a=b=0$, then $\mathcal{F}_{A,B}^{\mathcal{P}(M)}=\mathcal{T}_{n,d}(\mathcal{P}(M))\not=\mathcal{A}$. Otherwise, suppose at least one of $a,b$ is nonzero. 
	If we define  $\mathbf{T}$ by 
	 \[
	 T_j=b I_d, \quad T_{n-j}=a I_d,
	 \]
	and $T_i=0$ for other values of $T$, Then $\mathbf{T}$ satisfies~\eqref{eq:definition of FAB}, and therefore $\mathbf{T}\in \mathcal{F}_{A,B}^{\mathcal{P}(M)}$. But obviously $\mathbf{T}\notin\mathcal{A}$, so $\mathcal{A}\not= \mathcal{F}_{A,B}^{\mathcal{P}(M)}$.	
\end{proof}

\section*{Acknowledgements}

The author is highly grateful to the Abdus Salam School of Mathematical Sciences GC University, Lahore (ASSMS) for providing financial assistance during the preparation of this paper.

\end{document}